\documentclass[reqno]{amsart}

\usepackage{amsfonts,color}
\usepackage{amssymb,amsmath}
\usepackage[all]{xy}
\usepackage[dvips]{graphics}

\makeatletter
\newcommand{\leqnomode}{\tagsleft@true}
\makeatother

 \newtheorem{theorem}{Theorem}
 \newtheorem{corollary}[theorem]{Corollary}
 \newtheorem{lemma}[theorem]{Lemma}
 \newtheorem{proposition}[theorem]{Proposition}
 \theoremstyle{definition}
 
 \theoremstyle{remark}
 \newtheorem{remark}[theorem]{Remark}

 \numberwithin{equation}{section}
 \numberwithin{theorem}{section}

\begin{document}

\title[Optimal extensions for $p$-th power factorable operators]
{Optimal extensions for $p$-th power factorable operators}

\author[O.\ Delgado]{O.\ Delgado}
\address{Departamento de Matem\'atica Aplicada I, E.\ T.\ S.\ de Ingenier\'ia
de Edificaci\'on, Universidad de Sevilla, Avenida de Reina Mercedes,
4 A,  Sevilla 41012, Spain}
\email{\textcolor[rgb]{0.00,0.00,0.84}{olvido@us.es}}

\author[E.\ A.\ S\'{a}nchez P\'{e}rez]{E.\ A.\ S\'{a}nchez P\'{e}rez}
\address{Instituto Universitario de Matem\'atica Pura y Aplicada,
Universitat Polit\`ecnica de Val\`encia, Camino de Vera s/n, 46022
Valencia, Spain.}
\email{\textcolor[rgb]{0.00,0.00,0.84}{easancpe@mat.upv.es}}

\subjclass[2010]{46E30, 46G10, 47B38.}

\keywords{Integration with respect to a vector measure defined on a
$\delta$-ring, optimal extension of an operator, $p$-th power
factorable operator, quasi-Banach function space.}

\thanks{The first author gratefully acknowledge the support of the Ministerio de Econom\'{\i}a y Competitividad
(project \#MTM2012-36732-C03-03) and the Junta de Andaluc\'{\i}a
(projects FQM-262 and FQM-7276), Spain.}
\thanks{The second author acknowledges with thanks the support of the Ministerio de Econom\'{\i}a y Competitividad
(project \#MTM2012-36740-C02-02), Spain.}

\date{\today}

\maketitle


\begin{abstract}
Let $X(\mu)$ be a function space related to a measure space
$(\Omega,\Sigma,\mu)$ with $\chi_\Omega\in X(\mu)$ and let $T\colon
X(\mu)\to E$ be a Banach space valued operator. It is known that if
$T$ is $p$-th power factorable then the largest function space to
which $T$ can be extended preserving $p$-th power factorability is
given by the space $L^p(m_T)$ of $p$-integrable functions with
respect to $m_T$, where $m_T\colon\Sigma\to E$ is the vector measure
associated to $T$ via $m_T(A)=T(\chi_A)$. In this paper we extend
this result by removing the restriction $\chi_\Omega\in X(\mu)$. In
this general case, by considering $m_T$ defined on a certain
$\delta$-ring, we show that the optimal domain for $T$ is the space
$L^p(m_T)\cap L^1(m_T)$. We apply the obtained results to the
particular case when $T$ is a map between sequence spaces defined by
an infinite matrix.
\end{abstract}



\section{Introduction}

Although the concept of $p$-th power factorable operator have
previously been used as a tool in operator theory, it was introduced
explicitly in \cite[\S\,5]{okada-ricker-sanchezperez}. Given a
measure space $(\Omega,\Sigma,\mu)$ and a Banach function space
$X(\mu)$ of ($\mu$-a.e.\ classes of) $\Sigma$-measurable functions
such that $\chi_\Omega\in X(\mu)$, for $1\le p<\infty$, a Banach
space valued operator $T\colon X(\mu)\to E$ is $p$-th power
factorable if there is a continuous extension of $T$ to the
$\frac{1}{p}$-th power space $X(\mu)^{\frac{1}{p}}$ of $X(\mu)$.
This is equivalent to the existence of a constant $C>0$ satisfying
that
$$
\Vert T(f)\Vert\le
C\,\Vert\,|f|^{\frac{1}{p}}\,\Vert_{X(\mu)}^p=C\,\Vert
f\Vert_{X(\mu)^{\frac{1}{p}}}
$$
for all $f\in X(\mu)$. The main characterization of this class of
operators establishes that any of them can be extended to an space
$L^p$ of a vector measure $m_T\colon\Sigma\to E$ associated to $T$
via $m_T(A)=T(\chi_A)$ and the extension is maximal. Note that the
condition $\chi_\Omega\in X(\mu)$ is necessary for a correct
definition of $p$-th power factorable operator (i.e.\ $X(\mu)\subset
X(\mu)^{\frac{1}{p}}$) and for $m_T$ to be well defined.

Several applications are shown also in
\cite[\S\,6-7]{okada-ricker-sanchezperez}, mainly in factorization
of operators through spaces $L^q(\mu)$ (Maurey-Rosenthal type
theorems) and in harmonic analysis (Fourier transform and
convolution operators). After that, $p$-th power factorable
operators have turned out to be a useful tool for the study of
different problems in mathematical analysis, regarding for example
Banach space interpolation theory
\cite{delcampo-fernandez-galdames-mayoral-naranjo}, differential
equations \cite{galdames1}, description of maximal domains for
several classes of operators \cite{galdames-sanchezperez1},
factorization of kernel operators \cite{galdames-sanchezperez2} or
adjoint operators \cite{galdames2}.

The requirement $\chi_\Omega\in X(\mu)$ excludes basic spaces as
$L^q(0,\infty)$ or $\ell^q$. Although these spaces can be
represented as spaces satisfying the needed requirement (for
instance $L^q(0,\infty)$ is isometrically isomorphic to
$L^q(e^{-x}dx)$ via the multiplication operator induced by
$e^{\frac{x}{q}}$), to use such a representation provides some kind
of factorization for $T$ but not genuine extensions.

The aim of this paper is to extend the results on maximal extensions
of $p$-th power factorable operators to quasi-Banach spaces $X(\mu)$
which do not necessary contain $\chi_\Omega$. Also we will consider
$p$ to be any positive number removing the restriction $p\ge1$. The
first problem is the definition of $p$-th power factorable operator,
as in general the containment $X(\mu)\subset X(\mu)^{\frac{1}{p}}$
does not hold. This can be solved by replacing
$X(\mu)^{\frac{1}{p}}$ by the sum $X(\mu)^{\frac{1}{p}}+X(\mu)$. The
second problem is the definition of the vector measure $m_T$
associated to $T$. The technique to overcome this obstacle consists
of considering $m_T$ defined on the $\delta$-ring
$\Sigma_{X(\mu)}=\big\{A\in\Sigma:\,\chi_A\in X(\mu)\big\}$ instead
of the $\sigma$-algebra $\Sigma$. We will see that actually no
topology is needed on $X(\mu)$ to extend $T\colon X(\mu)\to E$, it
suffices an ideal structure on $X(\mu)$ and a certain property on
$T$ which relates the $\mu$-a.e.\ pointwise order of $X(\mu)$ and
the weak topology of $E$. This property, called order-w continuity,
is the minimal condition for $m_T$ to be a vector measure.

The paper is organized as follows. Section \ref{SEC: Preliminaries}
is devoted to establish the notation and to state the results on
ideal function spaces, quasi-Banach function spaces and integration
with respect to a vector measure defined on a $\delta$-ring, which
will be use along this work. For the aim of completeness, we include
the proof of some relevant facts. In Section \ref{SEC:
orderwCont-ifs} we show that every order-w continuous operator $T$
defined on an ideal function space $X(\mu)$, can be extended to the
space $L^1(m_T)$ of integrable functions with respect to $m_T$ and
this space is the largest one to which $T$ can be extended as an
order-w continuous operator (Theorem \ref{THM:
OrderwContinuousFactorization}). Section \ref{SEC:
orderwCont-pthPowerFactor-ifs} deals with operators $T$ which are
$p$-th power factorable with an order-w continuous extension, that
is, there is an order-w continuous extension of $T$ to the space
$X(\mu)^{\frac{1}{p}}+X(\mu)$. We prove that the space $L^p(m_T)\cap
L^1(m_T)$ is the optimal domain for $T$ preserving the property of
being $p$-th power factorable with an order-w continuous extension
(Theorem \ref{THM: orderwCont-pthPowerFactor-Factorization}). In
Sections \ref{SEC: Continuos-quasiBfs} and \ref{SEC:
orderwCont-pthPowerFactor-quasiBfs} we endow $X(\mu)$ with a
topology (namely, $X(\mu)$ will be a $\sigma$-order continuous
quasi-Banach function space) and consider $T$ to be continuous.
Results on maximal extensions analogous to the ones of the previous
sections are obtain for continuity instead of order-w continuity
(Theorems \ref{THM: ContinuousFactorization} and \ref{THM:
Cont-pthPowerFactor-Factorization}). Finally, as an application of
our results, in the last section we study when an infinite matrix of
real numbers defines a continuous linear operator from $\ell^p$ into
any given sequence space.


\section{Preliminaries}\label{SEC: Preliminaries}

\subsection{Ideal function spaces}

Let $(\Omega,\Sigma)$ be a fixed measurable space. For a measure
$\mu\colon\Sigma\to[0,\infty]$, we denote by $L^0(\mu)$ the space of
all ($\mu$-a.e.\ classes of) $\Sigma$-measurable real valued
functions on $\Omega$. Given two set functions
$\mu,\lambda\colon\Sigma\to[0,\infty]$ we will write $\lambda\ll\mu$
if $\mu(A)=0$ implies $\lambda(A)=0$. We will say that $\mu$ and
$\lambda$ are \emph{equivalent} if $\lambda\ll\mu$ and
$\mu\ll\lambda$. In the case when $\mu$ and $\lambda$ are two
measures with $\lambda\ll\mu$, the map $[i]\colon L^0(\mu)\to
L^0(\lambda)$ which takes a $\mu$-a.e.\ class in $L^0(\mu)$
represented by $f$ into the $\lambda$-a.e.\ class represented by the
same $f$, is a well defined linear map. In order to simplify
notation $[i](f)$ will be denoted again as $f$. Note that if
$\lambda$ and $\mu$ are equivalent then $L^0(\mu)=L^0(\lambda)$ and
$[i]$ is the identity map $i$.

An \emph{ideal function space} (briefly, i.f.s.) is a vector space
$X(\mu)\subset L^0(\mu)$ satisfying that if $f\in X(\mu)$ and $g\in
L^0(\mu)$ with $|g|\le|f|$ $\mu$-a.e.\ then $g\in X(\mu)$. We will
say that $X(\mu)$ has the \emph{$\sigma$-property} if there exists
$(\Omega_n)\subset\Sigma$ such that $\Omega=\cup \Omega_n$ and
$\chi_{\Omega_n}\in X(\mu)$ for all $n$. For instance, this happens
if there is some $g\in X(\mu)$ with $g>0$ $\mu$-a.e.

\begin{lemma}\label{LEM: Xmu-a.e.denseL0}
Let $X(\mu)$ be an i.f.s.\ satisfying the $\sigma$-property. For
every $\Sigma$--measurable function $f\colon\Omega\to[0,\infty)$
there exists $(f_n)\subset X(\mu)$ such that $0\le f_n\uparrow f$
pointwise.
\end{lemma}

\begin{proof}
Let $(\Omega_n)\subset\Sigma$ be the sequence given by the
$\sigma$-property of $X(\mu)$ and let $f\colon\Omega\to[0,\infty)$
be a $\Sigma$--measurable function. Taking
$A_n=\cup_{j=1}^n\Omega_j\cap \big\{\omega\in\Omega:\, f(\omega)\le
n\big\}$, we have that $f_n=f\chi_{A_n}\in X(\mu)$, as $0\le f_n\le
n\chi_{\cup_{j=1}^n\Omega_j}$ pointwise, and that $f_n\uparrow f$
pointwise.
\end{proof}

The \emph{sum} of two i.f.s.' $X(\mu)$ and $Y(\mu)$ is the space
defined as
$$
X(\mu)+Y(\mu)=\big\{f\in L^0(\mu):\, f=f_1+f_2 \ \
\mu\textnormal{-a.e.}, \ f_1\in X(\mu), \ f_2\in Y(\mu)\big\}.
$$

\begin{proposition}\label{PROP: X+Y-i.f.s.}
The sum $X(\mu)+Y(\mu)$ of two i.f.s.' is an i.f.s.
\end{proposition}

\begin{proof}
Let $f\in X(\mu)+Y(\mu)$ and $g\in L^0(\mu)$ be such that
$|g|\le|f|$ $\mu$--a.e. Write $f=f_1+f_2$ $\mu$-a.e.\ with $f_1\in
X(\mu)$ and $f_2\in Y(\mu)$ and denote
$A=\big\{\omega\in\Omega:\,|g(\omega)|\le|f_1(\omega)|\big\}$.
Taking $h_1=|g|\chi_A+ |f_1|\chi_{\Omega\backslash A}$ and
$h_2=(|g|-|f_1|)\chi_{\Omega\backslash A}$, we have that $|g|=h_1+
h_2$ with $h_1\in X(\mu)$ as $0\le h_1\le |f_1|$ pointwise and
$h_2\in Y(\mu)$ as $0\le h_2\le |f_2|$ $\mu$-a.e. Now, denote
$B=\big\{\omega\in\Omega:\,g(\omega)\ge0\big\}$ and take
$g_1=h_1\big(\chi_B-\chi_{\Omega\backslash B})$ and
$g_2=h_2\big(\chi_B-\chi_{\Omega\backslash B})$. Then, $g=g_1+g_2$
with $g_1\in X(\mu)$ as $|g_1|=h_1$ and $g_2\in Y(\mu)$ as
$|g_2|=h_2$. So, $g\in X(\mu)+Y(\mu)$.
\end{proof}

Let $p\in(0,\infty)$. The \emph{$p$-power} of an i.f.s.\ $X(\mu)$ is
the i.f.s.\ defined as
$$
X(\mu)^p=\big\{f\in L^0(\mu):\, |f|^p\in X(\mu)\big\}.
$$

\begin{lemma}\label{LEM: XsXtSubsetXr}
Let $X(\mu)$ be an i.f.s. For $s,t\in(0,\infty)$ and
$\frac{1}{r}=\frac{1}{s}+\frac{1}{t}$, it follows that if $f\in
X(\mu)^s$ and $g\in X(\mu)^t$ then $fg\in X(\mu)^r$. In particular,
if $\chi_\Omega\in X(\mu)$ then $X(\mu)^q\subset X(\mu)^p$ for all
$0<p<q<\infty$.
\end{lemma}

\begin{proof}
For the first part only note that for every $a,b>0$ it follows
\begin{equation}\label{EQ: XsXtSubsetXr}
a^rb^r\le\frac{r}{s}\,a^s+\frac{r}{t}\,b^t.
\end{equation}
For the second part take $r=p$, $s=q$ and $t=\frac{pq}{q-p}$. Then,
if $f\in X(\mu)^q$, since $\chi_\Omega\in X(\mu)^t$, we have that
$f=f\chi_\Omega\in X(\mu)^p$.
\end{proof}

Recall that a \emph{quasi-norm} on a real vector space $X$ is a
non-negative real map $\Vert\cdot\Vert_X$ on $X$ satisfying
\begin{itemize}\setlength{\leftskip}{-2ex}\setlength{\itemsep}{1ex}
\item[(i)] $\Vert x\Vert_X=0$ if and only if $x=0$,

\item[(ii)] $\Vert \alpha x\Vert_X=\vert\alpha\vert\cdot\Vert x\Vert_X$
for all $\alpha\in\mathbb{R}$ and $x\in X$, and

\item[(iii)] there exists a constant $K\ge1$ such that $\Vert
x+y\Vert_X\le K(\Vert x\Vert_X+\Vert y\Vert_X)$ for all $x,y\in X$.
\end{itemize}
A quasi-norm $\Vert\cdot\Vert_X$ induces a metric topology on $X$ in
which a sequence $(x_n)$ converges to $x$ if and only if $\Vert
x-x_n\Vert_X\to0$. If $X$ is complete under this topology then it is
called a \emph{quasi-Banach space} (\emph{Banach space} if $K=1$). A
linear map $T\colon X\to Y$ between quasi-Banach spaces is
continuous if and only if there exists a constant $M>0$ such that
$\Vert T(x)\Vert_Y\le M\Vert x\Vert_X$ for all $x\in X$. For issues
related to quasi-Banach spaces see \cite{kalton-peck-roberts}.

A \emph{quasi-Banach function space} (quasi-B.f.s.\ for short) is a
i.f.s.\ $X(\mu)$ which is also a quasi-Banach space with a
quasi-norm $\Vert\cdot\Vert_{X(\mu)}$ \emph{compatible with the
$\mu$-a.e.\ pointwise order}, that is, if $f,g\in X(\mu)$ are such
that $|f|\le |g|$ $\mu$-a.e.\ then $\Vert f\Vert_{X(\mu)}\le\Vert
g\Vert_{X(\mu)}$. When the quasi-norm is a norm, $X(\mu)$ is called
a \emph{Banach function space} (B.f.s.). Note that every
quasi-B.f.s.\ is a quasi-Banach lattice for the $\mu$-a.e.\
pointwise order satisfying that if $f_n\to f$ in quasi-norm then
there exists a subsequence $f_{n_j}\to f$ $\mu$-a.e. Also note that
every positive linear operator between quasi-Banach lattices is
continuous, see the argument given in
\cite[p.\,2]{lindenstrauss-tzafriri} for Banach lattices which can
be adapted for quasi-Banach spaces. Then all
\lq\lq\,inclusions\rq\rq\ of the type $[i]$ between quasi-B.f.s.'
are continuous.

A quasi-B.f.s.\ $X(\mu)$ is said to be \emph{$\sigma$-order
continuous} if for every $(f_n)\subset X(\mu)$ with $f_n\downarrow0$
$\mu$-a.e.\ it follows that $\Vert f_n\Vert_X\downarrow0$.

It is routine to check that the intersection $X(\mu)\cap Y(\mu)$ of
two quasi-B.f.s.'\ (B.f.s.')\ $X(\mu)$ and $Y(\mu)$ is a
quasi-B.f.s.\ (B.f.s.)\ endowed with the quasi-norm (norm)
$$
\Vert f\Vert_{X(\mu)\cap Y(\mu)}=\max\big\{\Vert
f\Vert_{X(\mu)},\Vert f\Vert_{Y(\mu)}\big\}.
$$
Moreover, if $X(\mu)$ and $Y(\mu)$ are $\sigma$-order continuous
then $X(\mu)\cap Y(\mu)$ is $\sigma$-order continuous.

\begin{proposition}\label{PROP: X+Y-qBfs}
The sum $X(\mu)+Y(\mu)$ of two quasi-B.f.s.'\ (B.f.s.')\ $X(\mu)$
and $Y(\mu)$ is a quasi-B.f.s.\ (B.f.s.)\ endowed with the
quasi-norm (norm)
$$
\Vert f\Vert_{X(\mu)+Y(\mu)}=\inf\big(\Vert f_1\Vert_{X(\mu)}+\Vert
f_2\Vert_{Y(\mu)}\big),
$$
where the infimum is taken over all possible representations
$f=f_1+f_2$ $\mu$-a.e.\ with $f_1\in X(\mu)$ and $f_2\in Y(\mu)$.
Moreover, if $X(\mu)$ and $Y(\mu)$ are $\sigma$-order continuous
then $X(\mu)+Y(\mu)$ is also $\sigma$-order continuous.
\end{proposition}

\begin{proof}
From Proposition \ref{PROP: X+Y-i.f.s.} we have that $X(\mu)+Y(\mu)$
is a i.f.s. Even more, looking at the proof we see that for every
$f\in X(\mu)+Y(\mu)$ and $g\in L^0(\mu)$ with $|g|\le|f|$
$\mu$-a.e., if $f=f_1+f_2$ $\mu$-a.e.\ with $f_1\in X(\mu)$ and
$f_2\in Y(\mu)$ then there exist $g_1\in X(\mu)$ and $g_2\in Y(\mu)$
such that $|g_i|\le|f_i|$ $\mu$-a.e.\ and $g=g_1+g_2$. Then,
$$
\Vert g\Vert_{X(\mu)+Y(\mu)}\le \Vert g_1\Vert_{X(\mu)}+\Vert
g_2\Vert_{Y(\mu)}\le \Vert f_1\Vert_{X(\mu)}+\Vert f_2\Vert_{Y(\mu)}
$$
and so, taking infimum over all possible representations $f=f_1+f_2$
$\mu$-a.e.\ with $f_1\in X(\mu)$ and $f_2\in Y(\mu)$, it follows
that $\Vert g\Vert_{X(\mu)+Y(\mu)}\le\Vert f\Vert_{X(\mu)+Y(\mu)}$.
Hence, $\Vert\cdot\Vert_{X(\mu)+Y(\mu)}$ is compatible with the
$\mu$-a.e.\ pointwise order.

The proof of the fact that $\Vert\cdot\Vert_{X(\mu)+Y(\mu)}$ is a
quasi-norm for which $X(\mu)+Y(\mu)$ is complete is similar to the
one given in \cite[\S\,3, Theorem 1.3]{bennett-sharpley} for
compatible couples of Banach spaces.

Suppose that $X(\mu)$ and $Y(\mu)$ are $\sigma$-order continuous.
Let $(f_n)\subset X(\mu)+Y(\mu)$ be such that $f_n\downarrow0$
$\mu$-a.e. Consider $f_1=g+h$ $\mu$-a.e.\ with $g\in X(\mu)$ and
$h\in Y(\mu)$. We can rewrite $f_1=f_1^1+f_1^2$ with $f_1^1\in
X(\mu)$, $f_1^2\in Y(\mu)$ and $f_1^1,f_1^2\ge0$ $\mu$-a.e. This can
be done by taking $A=\big\{\omega\in\Omega:\,
f_1(\omega)\le|g(\omega)|\big\}$,
$f_1^1=f_1\chi_A+|g|\chi_{\Omega\backslash A}$ and
$f_1^2=(f_1-|g|)\chi_{\Omega\backslash A}$. Note that $f_1^1\in
X(\mu)$ as $0\le f_1^1\le|g|$ $\mu$-a.e.\ and $f_1^2\in Y(\mu)$ as
$0\le f_1^2\le|h|$ $\mu$-a.e. Since $0\le f_2\le f_1$ $\mu$-a.e.,
looking again at the proof of Proposition \ref{PROP: X+Y-i.f.s.} we
see that there exist $f_2^1\in X(\mu)$ and $f_2^2\in Y(\mu)$ such
that $0\le f_2^i\le f_1^i$ $\mu$-a.e.\ and $f_2=f_2^1+f_2^2$
$\mu$-a.e. By induction we construct two $\mu$-a.e.\ pointwise
decreasing sequences of positive functions $(f_n^1)\subset X(\mu)$
and $(f_n^2)\subset Y(\mu)$ such that $f_n=f_n^1+f_n^2$. Note that
$f_n^i\downarrow0$ $\mu$-a.e.\ as $0\le f_n^i\le f_n$ $\mu$-a.e.
Then, since $X(\mu)$ and $Y(\mu)$ are $\sigma$-order continuous, we
have that
$$
\Vert f_n\Vert_{X(\mu)+Y(\mu)}\le\Vert f_n^1\Vert_{X(\mu)}+\Vert
f_n^2\Vert_{Y(\mu)}\to 0.
$$
\end{proof}

Let $p\in(0,\infty)$. The $p$-power $X(\mu)^p$ of a quasi-B.f.s.\
$X(\mu)$ is a quasi-B.f.s.\ endowed with the quasi-norm
$$
\Vert f\Vert_{X(\mu)^p}=\Vert \,|f|^p\,\Vert_{X(\mu)}^{\frac{1}{p}}.
$$
Moreover, $X(\mu)^p$ is $\sigma$-order continuous whenever $X(\mu)$
is so. Note that in the case when $X(\mu)$ is a B.f.s.\ and $p\ge1$
it follows that $\Vert\cdot\Vert_{X(\mu)^p}$ is a norm and so
$X(\mu)^p$ is a B.f.s. An exhaustive study of the space $X(\mu)^p$
can be found in \cite[\S\,2.2]{okada-ricker-sanchezperez} for the
case when $\mu$ is finite and $\chi_\Omega\in X(\mu)$. This study
can be extended to our general case adapting the arguments with the
natural modifications (note that our $p$-powers here are the
$\frac{1}{p}$-th powers there).

\subsection{Integration with respect to a vector measure defined on a
$\delta$-ring}

Let $\mathcal{R}$ be a \emph{$\delta$-ring} of subsets of a set
$\Omega$, that is, a ring closed under countable intersections.
Measurability will be considered with respect to the
$\sigma$-algebra $\mathcal{R}^{loc}$ of all subsets $A$ of $\Omega$
such that $A\cap B\in\mathcal{R}$ for all $B\in\mathcal{R}$. Let us
write $\mathcal{S}(\mathcal{R})$ for the space of all
\emph{$\mathcal{R}$-simple} functions, that is, simple functions
with support in $\mathcal{R}$.

A set function $m\colon\mathcal{R}\to E$ with values in a Banach
space $E$ is said to be a \emph{vector measure} if $\sum m(A_n)$
converges to $m(\cup A_n)$ in $E$ for every sequence of pairwise
disjoint sets $(A_n)\subset\mathcal{R}$ with $\cup
A_n\in\mathcal{R}$.

Consider first a \emph{real measure} $\lambda\colon\mathcal{R}\to
\mathbb{R}$. The \emph{variation} of $\lambda$ is the measure
$|\lambda|\colon\mathcal{R}^{loc}\to[0,\infty]$ defined as
$$
|\lambda|(A)=\sup\Big\{\sum|\lambda(A_j)|:\, (A_j) \textnormal{
finite disjoint sequence in } \mathcal{R}\cap 2^A\Big\}.
$$
Note that $|\lambda|$ is finite on $\mathcal{R}$. The space
$L^1(\lambda)$ of integrable functions with respect to $\lambda$ is
defined as the classical space $L^1(|\lambda|)$. The integral with
respect to $\lambda$ of
$\varphi=\sum_{j=1}^n\alpha_j\chi_{A_j}\in\mathcal{S}(\mathcal{R})$
over $A\in\mathcal{R}^{loc}$ is defined in the natural way by
$\int_A\varphi\,d\lambda=\sum_{j=1}^n\alpha_j\lambda(A_j\cap A)$.
The space $\mathcal{S}(\mathcal{R})$ is dense in $L^1(\lambda)$,
allowing to define the integral of $f\in L^1(\lambda)$ over
$A\in\mathcal{R}^{loc}$ as
$\int_Af\,d\lambda=\lim\int_A\varphi_n\,d\lambda$ for any sequence
$(\varphi_n)\subset\mathcal{S}(\mathcal{R})$ converging to $f$ in
$L^1(\lambda)$.

Let now $m\colon\mathcal{R}\to E$ be a vector measure. The
\emph{semivariation} of $m$ is the set function $\Vert
m\Vert\colon\mathcal{R}^{loc}\to[0,\infty]$ defined by
$$
\Vert m\Vert(A)=\sup_{x^*\in B_{E^*}}|x^*m|(A).
$$
Here, $B_{E^*}$ is the closed unit ball of the  dual space $E^*$ of
$E$ and $|x^*m|$ is the variation of the real measure $x^*m$ given
by the composition of $m$ with $x^*$. A set $A\in\mathcal{R}^{loc}$
is \emph{$m$-null} if $\Vert m\Vert(A)=0$, or equivalently, if
$m(B)=0$ for all $B\in\mathcal{R}\cap2^A$. From \cite[Theorem
3.2]{brooks-dinculeanu}, there always exists a measure
$\eta\colon\mathcal{R}^{loc}\to[0,\infty]$ equivalent to $\Vert
m\Vert$, that is $m$ and $\eta$ have the same null sets. Let us
denote $L^0(m)=L^0(\eta)$.

The space $L^1(m)$ of integrable functions with respect to $m$ is
defined as the space of functions $f\in L^0(m)$ satisfying that
\begin{itemize}\setlength{\leftskip}{-3ex}\setlength{\itemsep}{1ex}
\item[(i)] $f\in L^1(x^*m)$ for every $x^*\in E^*$,
and
\item[(ii)] for each $A\in \mathcal{R}^{loc}$ there exists $x_A \in E$ such that
$$
x^*(x_A)=\int_A f\,dx^*m,\ \ \textnormal{ for every } x^*\in E^*.
$$
\end{itemize}
The vector $x_A$ is unique and will be denoted by  $\int_A f\,dm$.
The space $L^1(m)$ is a $\sigma$-order continuous B.f.s.\ related to
the measure space $(\Omega,\mathcal{R}^{loc},\eta)$, with norm
$$
\Vert f\Vert_{L^1(m)}=\sup_{x^*\in B_{E^*}}\int_\Omega |f|\,d|x^*m|.
$$
Moreover $\mathcal{S}(\mathcal{R})$ is dense in $L^1(m)$. Note that
$\int_A\varphi\,dm=\sum_{j=1}^n\alpha_jm(A_j\cap A)$ for every
$\varphi=\sum_{j=1}^n\alpha_j\chi_{A_j}\in\mathcal{S}(\mathcal{R})$
and $A\in\mathcal{R}^{loc}$.

The integration operator $I_m\colon L^1(m)\to E$ defined by
$I_m(f)=\int_\Omega f\,dm$ is a continuous linear operator with
$\Vert I_m(f)\Vert_E\le\Vert f\Vert_{L^1(m)}$. Even more,
\begin{equation}\label{EQ: L1(m)-norm}
\frac{1}{2}\Vert f\Vert_{L^1(m)}\le\sup_{A\in\mathcal{R}}\Vert
I_m(f\chi_A)\Vert_E\le\Vert f\Vert_{L^1(m)}
\end{equation}
for all $f\in L^1(m)$.

Let $p\in(0,\infty)$. We denote by $L^p(m)$ the $p$-power of
$L^1(m)$, that is,
$$
L^p(m)=\big\{f\in L^0(m):\, |f|^p\in L^1(m)\big\}.
$$
Then $L^p(m)$ is a quasi-B.f.s.\ with the quasi-norm $\Vert
f\Vert_{L^p(m)}=\Vert\,|f|^p\,\Vert_{L^1(m)}^{1/p}$. In the case
when $p\ge1$, we have that $\Vert\cdot\Vert_{L^p(m)}$ is a norm and
so $L^p(m)$ is a B.f.s.

These and other issues concerning integration with respect to a
vector measure defined on a $\delta$-ring can be found in
\cite{lewis}, \cite{masani-niemi1}, \cite{masani-niemi2},
\cite{delgado1}, \cite{calabuig-juan-sanchezperez} and
\cite{calabuig-delgado-juan-sanchezperez}.


\section{Optimal domain for order-w continuous operators on a
i.f.s.}\label{SEC: orderwCont-ifs}

Let $X(\mu)$ be a i.f.s.\ satisfying the $\sigma$-property (recall:
$\Omega=\cup\Omega_n$ with $\chi_{\Omega_n}\in X(\mu)$ for all $n$)
and consider the $\delta$--ring
$$
\Sigma_{X(\mu)}=\big\{A\in\Sigma:\, \chi_A\in X(\mu)\big\}.
$$
The $\sigma$-property guarantees that
$\Sigma_{X(\mu)}^{loc}=\Sigma$. Given a Banach space valued linear
operator $T\colon X(\mu)\to E$, we define the finitely additive set
function $m_T\colon\Sigma_{X(\mu)}\to E$ by $m_T(A)=T(\chi_A)$.

We will say that $T$ is \emph{order-w continuous} if $T(f_n)\to
T(f)$ weakly in $E$ whenever $f_n,f\in X(\mu)$ are such that $0\le
f_n\uparrow f$ $\mu$--a.e.

\begin{proposition}\label{PROP: XsubsetL1mT}
If $T$ is order-w continuous, then $m_T$ is a vector measure
satisfying that $[i]\colon X(\mu)\to L^1(m_T)$ is well defined and
$T=I_{m_T}\circ[i]$.
\end{proposition}

\begin{proof}
Let $(A_n)\subset\Sigma_{X(\mu)}$ be a pairwise disjoint sequence
with $\cup A_n\in\Sigma_{X(\mu)}$. Since $T$ is order-w continuous,
for any subsequence $(A_{n_j})$ we have that
$$
\sum_{j=1}^N m_T(A_{n_j})=T(\chi_{\cup_{j=1}^NA_{n_j}})\to
T(\chi_{\cup A_{n_j}})=m_T(\cup A_{n_j})
$$
weakly in $E$. From the Orlicz-Pettis theorem (see \cite[Corollary
I.4.4]{diestel-uhl}) it follows that $\sum m_T(A_n)$ is
unconditionally convergent in norm to $m_T(\cup A_n)$. Thus, $m_T$
is a vector measure.

Note that $\Vert m_T\Vert\ll\mu$ and so $[i]\colon L^0(\mu)\to
L^0(m_T)$ is well defined. Also, note that for every
$\varphi\in\mathcal{S}(\Sigma_{X(\mu)})$ we have that
$I_{m_T}(\varphi)=T(\varphi)$.

Let $f\in X(\mu)$ be such that $f\ge0$ $\mu$-a.e.\ and take a
sequence of $\Sigma$-simple functions $0\le\varphi_n\uparrow f$
$\mu$-a.e. For each $n$ we can write
$\varphi_n=\sum_{j=1}^m\alpha_j\chi_{A_j}$ with
$(A_j)_{j=1}^m\subset\Sigma$ being a pairwise disjoint sequence and
$\alpha_j>0$ for all $j$. Since
$\chi_{A_j}\le\alpha_j^{-1}\varphi_n\le\alpha_j^{-1}f$ $\mu$-a.e.,
we have that $\chi_{A_j}\in X(\mu)$ and so
$\varphi_n\in\mathcal{S}(\Sigma_{X(\mu)})$. Fix $x^*\in E^*$. For
every $A\in\Sigma$ it follows that $x^*T(\varphi_n\chi_A)\to
x^*T(f\chi_A)$ as $T$ is order-w continuous. Note that
$x^*T(\varphi_n\chi_A)=\int_A\varphi_n\,dx^*m_T$ and that
$0\le\varphi_n\uparrow f$ $x^*m_T$-a.e.\ as $|x^*m_T|\ll \Vert
m_T\Vert\ll\mu$. From \cite[Proposition 2.3]{delgado1}, we have that
$f\in L^1(x^*m_T)$ and
$$
\int_Af\,dx^*m_T=\lim_{n\to\infty}\int_A\varphi_n\,dx^*m_T=
\lim_{n\to\infty}x^*T(\varphi_n\chi_A)=x^*T(f\chi_A).
$$
Therefore, $f\in L^1(m_T)$ and $I_{m_T}(f)=T(f)$.

For a general $f\in X(\mu)$, the result follows by taking the
positive and negative parts of $f$.
\end{proof}

For the case when $X(\mu)$ is a B.f.s., Proposition \ref{PROP:
XsubsetL1mT} and the next Theorem \ref{THM:
OrderwContinuousFactorization} can be deduced from \cite[Proposition
2.3]{delgado2} and \cite[Proposition
4]{calabuig-delgado-sanchezperez}. The proofs given here are more
direct and are valid for general i.f.s.'.

\begin{theorem}\label{THM: OrderwContinuousFactorization}
Suppose that $T$ is order-w continuous. Then, $T$ factors as
\begin{equation}\label{EQ: OrderwContinuousFactorization1}
\begin{split}
\xymatrix{
X(\mu) \ar[rr]^{T} \ar@{.>}[dr]_(.45){[i]} & & E\\
& L^1(m_T) \ar@{.>}[ur]_(.55){I_{m_T}}}
\end{split}
\end{equation}
with $I_{m_T}$ being order-w continuous. Moreover, the factorization
is \emph{optimal} in the sense:
\medskip
$$
\left.\begin{minipage}{6.5cm} \textnormal{\it If $Z(\xi)$ is a
i.f.s.\ such that $\xi\ll\mu$ and} \leqnomode
\begin{equation}\label{EQ: OrderwContinuousFactorization2}
\begin{split}
\xymatrix{
X(\mu) \ar[rr]^{T} \ar@{.>}[dr]_(.45){[i]} & & E\\
& Z(\xi) \ar@{.>}[ur]_(.5){S}}
\end{split}
\end{equation}
\textnormal{\it with $S$ being an order-w continuous linear
operator}
\end{minipage}\ \right\} \ \Longrightarrow \ \
\begin{minipage}{6cm}
\textnormal{\it $[i]\colon Z(\xi)\to L^1(m_T)$ is well} \vspace{1mm} \\
\textnormal{\it defined and $S=I_{m_T}\circ[i]$.}
\end{minipage}
$$
\end{theorem}

\begin{proof}
The factorization \eqref{EQ: OrderwContinuousFactorization1} follows
from Proposition \ref{PROP: XsubsetL1mT}. Note that the integration
operator $I_{m_T}\colon L^1(m_T)\to E$ is order-w continuous, as it
is continuous and $L^1(m_T)$ is $\sigma$-order continuous.

Let $Z(\xi)$ satisfy \eqref{EQ: OrderwContinuousFactorization2}. In
particular, $Z(\xi)$ satisfies the $\sigma$-property, as if
$\chi_{A}\in X(\mu)$ then $\chi_A\in Z(\xi)$. From Proposition
\ref{PROP: XsubsetL1mT} applied to the operator $S\colon Z(\xi)\to
E$, we have that $[i]\colon Z(\xi)\to L^1(m_S)$ is well defined and
$S=I_{m_S}\circ[i]$. Note that
$\Sigma_{X(\mu)}\subset\Sigma_{Z(\xi)}$ and
$m_S(A)=S(\chi_A)=T(\chi_A)=m_T(A)$ for all $A\in\Sigma_{X(\mu)}$,
that is, $m_T$ is the restriction of $m_S\colon\Sigma_{Z(\xi)}\to E$
to $\Sigma_{X(\mu)}$. Then, from \cite[Lemma
3]{calabuig-delgado-sanchezperez}, it follows that
$L^1(m_S)=L^1(m_T)$ and $I_{m_S}=I_{m_T}$.
\end{proof}

We can rewrite Theorem \ref{THM: OrderwContinuousFactorization} in
terms of optimal domain.

\begin{corollary}
Suppose that $T$ is order-w continuous. Then $L^1(m_T)$ is the
largest i.f.s.\ to which $T$ can be extended as an order-w
continuous operator still with values in $E$. Moreover, the
extension of $T$ to $L^1(m_T)$ is given by the integration operator
$I_{m_T}$.
\end{corollary}


\section{Optimal domain for $p$-th power factorable operators on a
i.f.s.\ with an order-w continuous extension} \label{SEC:
orderwCont-pthPowerFactor-ifs}

Let $X(\mu)$ be a i.f.s.\ satisfying the $\sigma$-property and let
$T\colon X(\mu)\to E$ be a linear operator with values in a Banach
space $E$.

For $p\in(0,\infty)$, we call $T$ \emph{$p$-th power factorable with
an order-w continuous extension} if there is an order-w continuous
linear extension of $T$ to $X(\mu)^{\frac{1}{p}}+X(\mu)$, i.e. $T$
factors as
$$
\xymatrix{
X(\mu) \ar[rr]^{T} \ar@{.>}[dr]_(0.4){i} & & E\\
& X(\mu)^{\frac{1}{p}}+X(\mu) \ar@{.>}[ur]_(0.6){S}}
$$
with $S$ being an order-w continuous linear operator.

Note that in the case when $\chi_\Omega\in X(\mu)$, from Lemma
\ref{LEM: XsXtSubsetXr}, if $1<p$ we have that $X(\mu)\subset
X(\mu)^{\frac{1}{p}}$ and so
$X(\mu)^{\frac{1}{p}}+X(\mu)=X(\mu)^{\frac{1}{p}}$. Similarly, if
$p\le1$ then $X(\mu)^{\frac{1}{p}}+X(\mu)=X(\mu)$, but hence to say
that $T$ is $p$-th power factorable with an order-w continuous
extension is just to say that $T$ is order-w continuous.

\begin{proposition}\label{PROP: Xsubset(LpCapL1)mT}
The following statements are equivalent:
\begin{itemize}\setlength{\leftskip}{-3ex}\setlength{\itemsep}{1ex}
\item[(a)] $T$ is $p$-th power factorable with an order-w continuous extension.

\item[(b)] $T$ is order-w continuous and $[i]\colon
X(\mu)^{\frac{1}{p}}+X(\mu)\to L^1(m_T)$ is well defined.

\item[(c)] $T$ is order-w continuous and $[i]\colon X(\mu)\to L^p(m_T)\cap L^1(m_T)$ is well defined.
\end{itemize}
Moreover, if (a)-(c) holds, the extension of $T$ to
$X(\mu)^{\frac{1}{p}}+X(\mu)$ coincides with integration operator
$I_{m_T}\circ[i]$.
\end{proposition}

\begin{proof}
(a) $\Rightarrow$ (b) Note that $T$ is order-w continuous as it has
an order-w continuous extension. Let $S\colon
X(\mu)^{\frac{1}{p}}+X(\mu)\to E$ be an order-w continuous linear
operator extending $T$. Then, from Theorem \ref{THM:
OrderwContinuousFactorization}, it follows that $[i]\colon
X(\mu)^{\frac{1}{p}}+X(\mu)\to L^1(m_T)$ is well defined and
$S=I_{m_T}\circ[i]$.

(b) $\Leftrightarrow$ (c) Since $T$ is is order-w continuous, by
Proposition \ref{PROP: XsubsetL1mT} we always have that $[i]\colon
X(\mu)\to L^1(m_T)$ is well defined. Suppose that $[i]\colon
X(\mu)^{\frac{1}{p}}+X(\mu)\to L^1(m_T)$ is well defined. If $f\in
X(\mu)$, since $|f|^p\in X(\mu)^{\frac{1}{p}}\subset
X(\mu)^{\frac{1}{p}}+X(\mu)$, we have that $|f|^p\in L^1(m_T)$ and
so $f\in L^p(m_T)$. Then $f\in L^p(m_T)\cap L^1(m_T)$. Conversely,
suppose that $[i]\colon X(\mu)\to L^p(m_T)\cap L^1(m_T)$ is well
defined. Let $f\in X(\mu)^{\frac{1}{p}}+X(\mu)$ and write
$f=f_1+f_2$ $\mu$-a.e.\ with $f_1\in X(\mu)^{\frac{1}{p}}$ and
$f_2\in X(\mu)$. Since $|f_1|^{\frac{1}{p}}\in X(\mu)$ we have that
$|f_1|^{\frac{1}{p}}\in L^p(m_T)\cap L^1(m_T)\subset L^p(m_T)$ and
so $f_1\in L^1(m_T)$. Then, $f\in L^1(m_T)$ as $f_2\in L^1(m_T)$.

(b) $\Rightarrow$ (a) From Proposition \ref{PROP: XsubsetL1mT} and
since $[i]\colon X(\mu)^{\frac{1}{p}}+X(\mu)\to L^1(m_T)$ is well
defined, we have that the operator $I_{m_T}\circ[i]$ extends $T$ to
$X(\mu)^{\frac{1}{p}}+X(\mu)$. Moreover, the extension
$I_{m_T}\circ[i]\colon X(\mu)^{\frac{1}{p}}+X(\mu)\to E$ is order-w
continuous as the integration operator $I_{m_T}\colon L^1(m_T)\to E$
is so.
\end{proof}

In the case when $\chi_\Omega\in X(\mu)$ and $T$ is order-w
continuous, from Proposition \ref{PROP: XsubsetL1mT}, we have that
$\chi_\Omega\in L^1(m_T)$. So, from Lemma \ref{LEM: XsXtSubsetXr},
if $p>1$ then $L^p(m_T)\subset L^1(m_T)$ and hence $L^p(m_T)\cap
L^1(m_T)=L^p(m_T)$. If $p\le1$ then $L^p(m_T)\cap
L^1(m_T)=L^1(m_T)$, but hence, as commented before, $T$ being $p$-th
power factorable with an order-w continuous extension is just $T$
being order-w continuous.

\begin{theorem}\label{THM: orderwCont-pthPowerFactor-Factorization}
Suppose that $T$ is $p$-th power factorable with an order-w
continuous extension. Then, $T$ factors as
\begin{equation}\label{EQ: orderwCont-pthPowerFactor-Factorization1}
\begin{split}
\xymatrix{
X(\mu) \ar[rr]^{T} \ar@{.>}[dr]_(.4){[i]} & & E\\
& L^p(m_T)\cap L^1(m_T) \ar@{.>}[ur]_(.6){I_{m_T}}}
\end{split}
\end{equation}
with $I_{m_T}$ being $p$-th power factorable with an order-w
continuous extension. Moreover, the factorization is \emph{optimal}
in the sense:
\medskip
$$
\left.\begin{minipage}{6cm} \textnormal{\it If $Z(\xi)$ is a i.f.s.\
such that $\xi\ll\mu$ and} \leqnomode
\begin{equation}\label{EQ: orderwCont-pthPowerFactor-Factorization2}
\begin{split}
\xymatrix{
X(\mu) \ar[rr]^{T} \ar@{.>}[dr]_(.45){[i]} & & E\\
& Z(\xi) \ar@{.>}[ur]_{S}}
\end{split}
\end{equation}
\textnormal{\it with $S$ being a $p$-th power factorable linear
operator with an order-w continuous extension}
\end{minipage}\ \right\} \ \Longrightarrow \ \
\begin{minipage}{6cm}
\textnormal{\it $[i]\colon Z(\xi)\to L^p(m_T)\cap L^1(m_T)$} \vspace{1mm} \\
\textnormal{\it is well defined and $S=I_{m_T}\circ[i]$.}
\end{minipage}
$$
\end{theorem}

\begin{proof}
The factorization \eqref{EQ:
orderwCont-pthPowerFactor-Factorization1} follows from Propositions
\ref{PROP: XsubsetL1mT} and \ref{PROP: Xsubset(LpCapL1)mT}. Note
that $L^p(m_T)\cap L^1(m_T)$ satisfies the $\sigma$-property as
$X(\mu)$ does. Let us see that the operator $I_{m_T}\colon
L^p(m_T)\cap L^1(m_T)\to E$ is $p$-th power factorable with an
order-w continuous extension by ussing Proposition \ref{PROP:
Xsubset(LpCapL1)mT}.(c). This operator is order-w continuous as the
integration operator $I_{m_T}\colon L^1(m_T)\to E$ is so. On other
hand, since $\Sigma_{X(\mu)}\subset\Sigma_{L^p(m_T)\cap L^1(m_T)}$
and $m_{I_{m_T}}(A)=I_{m_T}(\chi_A)=T(\chi_A)=m_T(A)$ for all
$A\in\Sigma_{X(\mu)}$ (i.e.\ $m_T$ is the restriction of
$m_{I_{m_T}}\colon\Sigma_{L^p(m_T)\cap L^1(m_T)}\to E$ to
$\Sigma_{X(\mu)}$), from \cite[Lemma
3]{calabuig-delgado-sanchezperez}, it follows that
$L^1(m_{I_{m_T}})=L^1(m_T)$. Then,
$$
[i]\colon L^p(m_T)\cap L^1(m_T)\to L^p(m_{I_{m_T}})\cap
L^1(m_{I_{m_T}})=L^p(m_T)\cap L^1(m_T)
$$
is well defined.

Let $Z(\xi)$ satisfy \eqref{EQ:
orderwCont-pthPowerFactor-Factorization2}. In particular, $Z(\xi)$
has the $\sigma$-property. Applying Proposition \ref{PROP:
Xsubset(LpCapL1)mT} to the operator $S\colon Z(\xi)\to E$, we have
that $[i]\colon Z(\xi)\to L^p(m_S)\cap L^1(m_S)$ is well defined and
$S=I_{m_S}\circ[i]$. Since $\Sigma_{X(\mu)}\subset\Sigma_{Z(\xi)}$
and $m_S(A)=m_T(A)$ for all $A\in\Sigma_{X(\mu)}$, from \cite[Lemma
3]{calabuig-delgado-sanchezperez}, it follows that
$L^1(m_S)=L^1(m_T)$ and $I_{m_S}=I_{m_T}$.
\end{proof}

Rewriting Theorem \ref{THM: orderwCont-pthPowerFactor-Factorization}
in terms of optimal domain we obtain the following conclusion.

\begin{corollary}
Suppose that $T$ is $p$-th power factorable with an order-w
continuous extension. Then $L^p(m_T)\cap L^1(m_T)$ is the largest
i.f.s.\ to which $T$ can be extended as a $p$-th power factorable
operator with an order-w continuous extension, still with values in
$E$. Moreover, the extension of $T$ to $L^p(m_T)\cap L^1(m_T)$ is
given by the integration operator $I_{m_T}$.
\end{corollary}


\section{Optimal domain for continuous operators on a quasi-B.f.s.}
\label{SEC: Continuos-quasiBfs}

Let $X(\mu)$ be a quasi-B.f.s.\ satisfying the $\sigma$-property and
let $T\colon X(\mu)\to E$ be a linear operator with values in a
Banach space $E$.

\begin{theorem}\label{THM: ContinuousFactorization}
Suppose that $X(\mu)$ is $\sigma$-order continuous and $T$ is
continuous. Then, $T$ factors as
\begin{equation}\label{EQ: ContinuousFactorization1}
\begin{split}
\xymatrix{
X(\mu) \ar[rr]^{T} \ar@{.>}[dr]_(.45){[i]} & & E\\
& L^1(m_T) \ar@{.>}[ur]_(.55){I_{m_T}}}
\end{split}
\end{equation}
with $I_{m_T}$ being continuous. Moreover, the factorization is
\emph{optimal} in the sense:
\medskip
$$
\left.\begin{minipage}{6.6cm} \textnormal{\it If $Z(\xi)$ is a
$\sigma$-order continuous quasi-B.f.s.\ such that $\xi\ll\mu$ and}
\leqnomode
\begin{equation}\label{EQ: ContinuousFactorization2}
\begin{split}
\xymatrix{
X(\mu) \ar[rr]^{T} \ar@{.>}[dr]_(.45){[i]} & & E\\
& Z(\xi) \ar@{.>}[ur]_{S}}
\end{split}
\end{equation}
\textnormal{\it with $S$ being a continuous linear operator}
\end{minipage}\ \right\} \ \Longrightarrow \ \
\begin{minipage}{6cm}
\textnormal{\it $[i]\colon Z(\xi)\to L^1(m_T)$ is well} \vspace{1mm} \\
\textnormal{\it defined and $S=I_{m_T}\circ[i]$.}
\end{minipage}
$$
\end{theorem}

\begin{proof}
Since $X(\mu)$ is $\sigma$-order continuous and $T$ is continuous we
have that $T$ is order-w continuous and so the factorization
\eqref{EQ: ContinuousFactorization1} follows from Theorem \ref{THM:
OrderwContinuousFactorization}. Recall that $L^1(m_T)$ is
$\sigma$-order continuous and $I_{m_T}$ is continuous.

Let $Z(\xi)$ satisfy \eqref{EQ: ContinuousFactorization2}. In
particular $S$ is order-w continuous. From Theorem \ref{THM:
OrderwContinuousFactorization} we have that $[i]\colon Z(\xi)\to
L^1(m_T)$ is well defined and $S=I_{m_T}\circ[i]$.
\end{proof}

\begin{corollary}\label{COR: ContinuousFactorization}
Suppose that $X(\mu)$ is $\sigma$-order continuous and $T$ is
continuous. Then $L^1(m_T)$ is the largest $\sigma$-order continuous
quasi-B.f.s.\ to which $T$ can be extended as a continuous operator
still with values in $E$. Moreover, the extension of $T$ to
$L^1(m_T)$ is given by the integration operator $I_{m_T}$.
\end{corollary}


\section{Optimal domain for $p$-th power factorable operators on a
quasi-B.f.s.\ with a continuous extension}\label{SEC:
orderwCont-pthPowerFactor-quasiBfs}

Let $X(\mu)$ be a quasi-B.f.s.\ satisfying the $\sigma$-property and
let $T\colon X(\mu)\to E$ be a linear operator with values in a
Banach space $E$.

For $p\in(0,\infty)$, we call $T$ \emph{$p$-th power factorable with
a continuous extension} if there is a continuous linear extension of
$T$ to $X(\mu)^{\frac{1}{p}}+X(\mu)$, i.e. $T$ factors as
$$
\xymatrix{
X(\mu) \ar[rr]^{T} \ar@{.>}[dr]_(.4){i} & & E\\
& X(\mu)^{\frac{1}{p}}+X(\mu) \ar@{.>}[ur]_(.6){S}}
$$
with $S$ being a continuous linear operator.

Note that in the case when $\chi_\Omega\in X(\mu)$ and $1<p$, from
Lemma \ref{LEM: XsXtSubsetXr}, it follows that
$X(\mu)^{\frac{1}{p}}+X(\mu)=X(\mu)^{\frac{1}{p}}$. Then our
definition of $p$-th power factorable operator with a continuous
extension coincides with the one given in \cite[Definition
5.1]{okada-ricker-sanchezperez}. If $p\le1$, since
$X(\mu)^{\frac{1}{p}}+X(\mu)=X(\mu)$, to say that $T$ is $p$-th
power factorable with a continuous extension is just to say that $T$
is continuous.

\begin{proposition}\label{PROP: quasiBfsXsubset(LpCapL1)mT}
Suppose that $X(\mu)$ is $\sigma$-order continuous. Then, the
following statements are equivalent:
\begin{itemize}\setlength{\leftskip}{-3ex}\setlength{\itemsep}{1ex}
\item[(a)] $T$ is $p$-th power factorable with a continuous extension.

\item[(b)] $T$ is $p$-th power factorable with an order-w continuous extension.

\item[(c)] $T$ is order-w continuous and $[i]\colon
X(\mu)^{\frac{1}{p}}+X(\mu)\to L^1(m_T)$ is well defined.

\item[(d)] $T$ is order-w continuous and $[i]\colon X(\mu)\to L^p(m_T)\cap L^1(m_T)$ is well defined.

\item[(e)] There exists $C>0$ such that $\Vert
T(f)\Vert_E\le C\,\Vert f\Vert_{X(\mu)^{\frac{1}{p}}+X(\mu)}$ for
all $f\in X(\mu)$.
\end{itemize}
Moreover, if (a)-(e) holds, the extension of $T$ to
$X(\mu)^{\frac{1}{p}}+X(\mu)$ coincides with the integration
operator $I_{m_T}\circ[i]$.
\end{proposition}

\begin{proof}
(a) $\Rightarrow$ (b) Let $S\colon X(\mu)^{\frac{1}{p}}+X(\mu)\to E$
be a continuous linear operator extending $T$. From Proposition
\ref{PROP: X+Y-qBfs} we have that $X(\mu)^{\frac{1}{p}}+X(\mu)$ is
$\sigma$-order continuous and so $S$ is order-w continuous. Then,
$T$ is $p$-th power factorable with an order-w continuous extension.

(b) $\Leftrightarrow$ (c) $\Leftrightarrow$ (d) and the fact that
the extension of $T$ to $X(\mu)^{\frac{1}{p}}+X(\mu)$ coincides with
the integration operator $I_{m_T}\circ[i]$ follow from Proposition
\ref{PROP: Xsubset(LpCapL1)mT}.

(c) $\Rightarrow$ (e) The operator $[i]\colon
X(\mu)^{\frac{1}{p}}+X(\mu)\to L^1(m_T)$ is continuous as it is
positive. Then, there exists a constant $C>0$ satisfying that
$$
\Vert f\Vert_{L^1(m_T)}\le C\,\Vert
f\Vert_{X(\mu)^{\frac{1}{p}}+X(\mu)}
$$
for all $f\in X(\mu)^{\frac{1}{p}}+X(\mu)$. Since $I_{m_T}$ extends
$T$ to $L^1(m_T)$, it follows that
$$
\Vert T(f)\Vert_E=\Vert I_{m_T}(f)\Vert_E\le\Vert
f\Vert_{L^1(m_T)}\le C\,\Vert f\Vert_{X(\mu)^{\frac{1}{p}}+X(\mu)}
$$
for all $f\in X(\mu)$.

(e) $\Rightarrow$ (a) Let $0\le f\in X(\mu)^{\frac{1}{p}}+X(\mu)$.
From Lemma \ref{LEM: Xmu-a.e.denseL0}, there exists $(f_n)\subset
X(\mu)$ such that $0\le f_n\uparrow f$ $\mu$-a.e. Since
$X(\mu)^{\frac{1}{p}}+X(\mu)$ is $\sigma$-order continuous, it
follows that $f_n\to f$ in the quasi-norm of
$X(\mu)^{\frac{1}{p}}+X(\mu)$. Then, since
$$
\Vert T(f_n)-T(f_m)\Vert_E=\Vert T(f_n-f_m)\Vert_E\le C\,\Vert
f_n-f_m\Vert_{X(\mu)^{\frac{1}{p}}+X(\mu)},
$$
we have that $\big(T(f_n)\big)$ converges to some element $e\in E$.
Define $S(f)=e$. Note that if $(g_n)\subset X(\mu)$ is another
sequence such that $0\le g_n\uparrow f$ $\mu$-a.e., then
\begin{eqnarray*}
\Vert T(f_n)-T(g_n)\Vert_E & \le & C\,\Vert
f_n-g_n\Vert_{X(\mu)^{\frac{1}{p}}+X(\mu)}
\\ & \le & CK\Big(\Vert f_n-f\Vert_{X(\mu)^{\frac{1}{p}}+X(\mu)}
+\Vert f-g_n\Vert_{X(\mu)^{\frac{1}{p}}+X(\mu)}\Big),
\end{eqnarray*}
where $K$ is the constant satisfying the property (iii) of the
quasi-norm $\Vert\cdot\Vert_{X(\mu)^{\frac{1}{p}}+X(\mu)}$, and so
$S$ is well defined. Also note that
\begin{eqnarray*}
\Vert S(f)\Vert_E & \le & \Vert S(f)-T(f_n)\Vert_E+\Vert
T(f_n)\Vert_E
\\ & \le & \Vert S(f)-T(f_n)\Vert_E+C\,\Vert f_n\Vert_{X(\mu)^{\frac{1}{p}}+X(\mu)} \\ & \le & \Vert
S(f)-T(f_n)\Vert_E+C\,\Vert f\Vert_{X(\mu)^{\frac{1}{p}}+X(\mu)}
\end{eqnarray*}
for all $n\ge1$, and thus $\Vert S(f)\Vert_E\le C\,\Vert
f\Vert_{X(\mu)^{\frac{1}{p}}+X(\mu)}$.

For a general $f\in X(\mu)^{\frac{1}{p}}+X(\mu)$, define
$S(f)=S(f^+)-S(f^-)$ where $f^+$ and $f^-$ are the positive and
negative parts of $f$ respectively.  It follows that $S$ is linear
and $S(f)=T(f)$ for all $f\in X(\mu)$. Moreover, for every $f\in
X(\mu)^{\frac{1}{p}}+X(\mu)$ we have that
\begin{eqnarray*}
\Vert S(f)\Vert_E & \le & \Vert S(f^+)\Vert_E+\Vert S(f^-)\Vert_E \\
& \le & C\,\Vert f^+\Vert_{X(\mu)^{\frac{1}{p}}+X(\mu)}+C\,\Vert
f^-\Vert_{X(\mu)^{\frac{1}{p}}+X(\mu)} \\ & \le & 2C\,\Vert
f\Vert_{X(\mu)^{\frac{1}{p}}+X(\mu)}.
\end{eqnarray*}
an so $S$ is continuous. Hence, $T$ is $p$-th power factorable with
a continuous extension.
\end{proof}

In the case when $\mu$ is finite, $\chi_\Omega\in X(\mu)$ and
$p\ge1$, the equivalences (a) $\Leftrightarrow$ (c)
$\Leftrightarrow$ (d) $\Leftrightarrow$ (e) of Proposition
\ref{PROP: quasiBfsXsubset(LpCapL1)mT} are proved in \cite[Theorem
5.7]{okada-ricker-sanchezperez}. Here we have included a more detailed
proof for the general case.

\begin{theorem}\label{THM: Cont-pthPowerFactor-Factorization}
Suppose that $X(\mu)$ is $\sigma$-order continuous and $T$ is $p$-th
power factorable with a continuous extension. Then, $T$ factors
as
\begin{equation}\label{EQ: Cont-pthPowerFactor-Factorization1}
\begin{split}
\xymatrix{
X(\mu) \ar[rr]^{T} \ar@{.>}[dr]_(.4){[i]} & & E\\
& L^p(m_T)\cap L^1(m_T) \ar@{.>}[ur]_(.6){I_{m_T}}}
\end{split}
\end{equation}
with $I_{m_T}$ being $p$-th power factorable with a continuous
extension. Moreover, the factorization is \emph{optimal} in the
sense:
\medskip
$$
\left.\begin{minipage}{6.1cm} \textnormal{\it If $Z(\xi)$ is a
$\sigma$-order continuous quasi-B.f.s.\ such that $\xi\ll\mu$ and}
\leqnomode
\begin{equation}\label{EQ: Cont-pthPowerFactor-Factorization2}
\begin{split}
\xymatrix{
X(\mu) \ar[rr]^{T} \ar@{.>}[dr]_(.45){[i]} & & E\\
& Z(\xi) \ar@{.>}[ur]_(.5){S}}
\end{split}
\end{equation}
\textnormal{\it with $S$ being a $p$-th power factorable linear
operator with a continuous extension}
\end{minipage}\ \right\} \ \Longrightarrow \ \
\begin{minipage}{6cm}
\textnormal{\it $[i]\colon Z(\xi)\to L^p(m_T)\cap L^1(m_T)$} \vspace{1mm} \\
\textnormal{\it is well defined and $S=I_{m_T}\circ[i]$.}
\end{minipage}
$$
\end{theorem}

\begin{proof}
From Proposition \ref{PROP: quasiBfsXsubset(LpCapL1)mT} we have that
$T$ is $p$-th power factorable with an order-w continuous extension.
Then, from Theorem \ref{THM:
orderwCont-pthPowerFactor-Factorization}, the factorization
\eqref{EQ: Cont-pthPowerFactor-Factorization1} holds and
$I_{m_T}\colon L^p(m_T)\cap L^1(m_T)\to E$ is $p$-th power
factorable with an order-w continuous extension. Noting that the
space $L^p(m_T)\cap L^1(m_T)$ is $\sigma$-order continuous (as
$L^1(m_T)$ is so) and satisfies the $\sigma$-property (as $X(\mu)$
does), from Proposition \ref{PROP: quasiBfsXsubset(LpCapL1)mT} it
follows that $I_{m_T}\colon L^p(m_T)\cap L^1(m_T)\to E$ is $p$-th
power factorable with a continuous extension.

Let $Z(\xi)$ satisfy \eqref{EQ: Cont-pthPowerFactor-Factorization2},
in particular it satisfies the $\sigma$-property. Again Proposition
\ref{PROP: quasiBfsXsubset(LpCapL1)mT} gives that $S$ is $p$-th
power factorable with an order-w continuous extension. So, from
Theorem \ref{THM: orderwCont-pthPowerFactor-Factorization}, it
follows that $[i]\colon Z(\xi)\to L^p(m_T)\cap L^1(m_T)$ is well
defined and $S=I_{m_T}\circ[i]$.
\end{proof}

\begin{corollary}\label{COR: Cont-pthPowerFactor-Factorization}
Suppose that $X(\mu)$ is $\sigma$-order continuous and $T$ is $p$-th
power factorable with a continuous extension. Then $L^p(m_T)\cap
L^1(m_T)$ is the largest $\sigma$-order continuous quasi-B.f.s.\ to
which $T$ can be extended as a $p$-th power factorable operator with
a continuous extension, still with values in $E$. Moreover, the
extension of $T$ to $L^p(m_T)\cap L^1(m_T)$ is given by the
integration operator $I_{m_T}$.
\end{corollary}

In the case when $\mu$ is finite, $\chi_\Omega\in X(\mu)$ and
$p\ge1$, Corollary \ref{COR: Cont-pthPowerFactor-Factorization} is
proved in \cite[Theorem 5.11]{okada-ricker-sanchezperez}.


\section{Application: extension for operators defined on $\ell^1$}
\label{SEC: Applications}

Consider the measure space $(\mathbb{N},\mathcal{P}(\mathbb{N}),c)$
where $c$ is the counting measure on $\mathbb{N}$. Note that a
property holds $c$-a.e.\ if and only if it holds pointwise and that
the space $L^0(c)$ coincides with the space $\ell^0$ of all real
sequences. Consider the space $\ell^1=L^1(c)$, which is
$\sigma$-order continuous and has the $\sigma$-property. The
$\delta$-ring $\mathcal{P}(\mathbb{N})_{\ell^1}$ is just the set
$\mathcal{P}_F(\mathbb{N})$ of all finite subsets of $\mathbb{N}$.

Let $T\colon\ell^1\to E$ be a continuous linear operator with values
in a Banach space $E$. Denote $e_n=\chi_{\{n\}}$ and assume that
$T(e_n)\not=0$ for all $n$. This assumption seems to be natural
since if $T(e_n)=0$ then the $n$-th coordinate is not involved in
the action of $T$. Hence, the vector measure
$m_T\colon\mathcal{P}_F(\mathbb{N})\to E$ associated to $T$ by
$m_T(A)=T(\chi_A)$ is equivalent to $c$ and so
$L^1(m_T)\subset\ell^0$. We will write $\ell^1(m_T)=L^1(m_T)$.

\begin{remark}\label{REM: l1(mT)-OptimalDomain}
By Theorem \ref{THM: ContinuousFactorization} we have that $T$
can be extended as
$$
\xymatrix{
\ell^1 \ar[rr]^{T} \ar@{.>}[dr]_(.4){i} & & E\\
& \ell^1(m_T) \ar@{.>}[ur]_(.55){I_{m_T}}}
$$
and $\ell^1(m_T)$ is the largest $\sigma$-order continuous
quasi-B.f.s.\ to which $T$ can be extended as a continuous operator.
\end{remark}

Let $p>1$. We have that $T$ is $\frac{1}{p}$-th power factorable
with a continuous extension if there is an extension $S$ as
$$
\xymatrix{
\ell^1 \ar[rr]^{T} \ar@{.>}[dr]_(0.4){i} & & E\\
& \ell^p \ar@{.>}[ur]_(0.55){S}}
$$
with $S$ being a continuous linear operator. Note that $p\le1$ is
not considered as in this case $\ell^p\subset\ell^1$ and so the
extension of $T$ to the sum $\ell^p+\ell^1$ is just the same
operator $T$. Applying Proposition \ref{PROP:
quasiBfsXsubset(LpCapL1)mT} in the context of this section we obtain
the following result.

\begin{proposition}\label{PROP: T:l1->E}
The following statements are equivalent:
\begin{itemize}\setlength{\leftskip}{-3ex}\setlength{\itemsep}{1ex}
\item[(a)] $T$ is $\frac{1}{p}$-th power factorable with a continuous extension.

\item[(b)] $\ell^p\subset\ell^1(m_T)$.

\item[(c)] $\ell^1\subset\ell^{\frac{1}{p}}(m_T)\cap\ell^1(m_T)$.

\item[(d)] There exists $C>0$ such that
$$
\Big\Vert\sum_{j\in M}x_jT(e_j)\Big\Vert_E\le C\Big(\sum_{j\in
M}x_j^p\Big)^{\frac{1}{p}}
$$
for all $M\in\mathcal{P}_F(\mathbb{N})$ and $(x_j)_{j\in M}\subset
[0,\infty)$.
\end{itemize}
\end{proposition}

\begin{proof}
From Proposition \ref{PROP: quasiBfsXsubset(LpCapL1)mT}, we only
have to prove that condition (d) is equivalent to the following
condition:
\begin{itemize}\setlength{\leftskip}{-2ex}\setlength{\itemsep}{1ex}
\item[(d')] There exists $C>0$ such that $\Vert T(x)\Vert_E\le C\,\Vert x\Vert_{\ell^p}$
for all $x\in\ell^1$.
\end{itemize}

If (d') holds, we obtain (d) by taking in (d') the element
$x=\sum_{j\in M}x_je_j\in\ell^1$ for every
$M\in\mathcal{P}_F(\mathbb{N})$ and $(x_j)_{j\in M}\subset
[0,\infty)$.

Suppose that (d) holds. Let $0\le x=(x_n)\in\ell^1$ and take
$y^k=\sum_{j=1}^kx_je_j$. Since $y^k\uparrow x$ pointwise, $\ell^1$
is $\sigma$-order continuous and $T$ is continuous, we have that
$$
\Vert T(x)\Vert_E=\lim\Vert
T(y^k)\Vert_E=\lim\Big\Vert\sum_{j=1}^kx_jT(e_j)\Big\Vert_E\le
C\,\lim\Big(\sum_{j=1}^kx_j^p\Big)^{\frac{1}{p}}=C\,\Vert
x\Vert_{\ell^p}.
$$
For a general $x\in\ell^1$, (d') follows by taking the positive and
negative parts of $x$.
\end{proof}

\begin{remark}\label{REM: l(1/p)(mT)Capl1(mT)-OptimalDomain}
Note that if $T$ is $\frac{1}{p}$-th power factorable with a
continuous extension then the integration operator $I_{m_T}$ extends
$T$ to $\ell^p$ and, from Theorem \ref{THM:
Cont-pthPowerFactor-Factorization}, $T$ factors optimally as
$$
\xymatrix{
\ell^1 \ar[rr]^{T} \ar@{.>}[dr]_(.35){i} & & E\\
& \ell^{\frac{1}{p}}(m_T)\cap\ell^1(m_T) \ar@{.>}[ur]_(.6){I_{m_T}}}
$$
with $I_{m_T}$ being $\frac{1}{p}$-th power factorable with a
continuous extension.
\end{remark}

Now a natural question arises: When
$\ell^{\frac{1}{p}}(m_T)\cap\ell^1(m_T)$ is equal to
$\ell^{\frac{1}{p}}(m_T)$ or $\ell^1(m_T)$? For asking this question
we introduce the following class of operators.

Let $0<r<\infty$. We say that $T$ is \emph{$r$-power dominated} if
there exists $C>0$ such that
$$
\Big\Vert\sum_{j\in M}x_j^r\,T(e_j)\Big\Vert_E^{\frac{1}{r}}\le
C\sup_{N\subset M}\Big\Vert\sum_{j\in N}x_jT(e_j)\Big\Vert_E
$$
for every $M\in\mathcal{P}_F(\mathbb{N})$ and $(x_j)_{j\in
M}\in[0,\infty)$. Note that in the case when $E$ is a Banach lattice
and $T$ is positive we have that
$$
\sup_{N\subset M}\Big\Vert\sum_{j\in
N}x_jT(e_j)\Big\Vert_E=\Big\Vert\sum_{j\in M}x_jT(e_j)\Big\Vert_E.
$$

\begin{lemma}\label{LEM: l1(mT)subsetlr(mT)}
The containment $\ell^1(m_T)\subset\ell^r(m_T)$ holds if and only if
$T$ is $r$-power dominated.
\end{lemma}

\begin{proof}
Suppose that $\ell^1(m_T)\subset\ell^r(m_T)$. Since the containment
is continuous (as it is positive), there exists $C>0$ such that
$\Vert x\Vert_{\ell^r(m_T)}\le C\,\Vert x\Vert_{\ell^1(m_T)}$ for
all $x\in\ell^1(m_T)$. For every $M\in \mathcal{P}_F(\mathbb{N})$
and $(x_j)_{j\in M}\in[0,\infty)$, we consider $x=\sum_{j\in
M}x_je_j\in\ell^1$. Noting that $x^r=\sum_{j\in
M}x_j^re_j\in\ell^1$, it follows that
\begin{eqnarray*}
\Big\Vert\sum_{j\in M}x_j^r\,T(e_j)\Big\Vert_E^{\frac{1}{r}} & = &
\Vert T(x^r)\Vert_E^{\frac{1}{r}}=\Vert
I_{m_T}(x^r)\Vert_E^{\frac{1}{r}}\le\Vert
x^r\Vert_{\ell^1(m_T)}^{\frac{1}{r}}=\Vert x\Vert_{\ell^r(m_T)}
\\ & \le & C\,\Vert x\Vert_{\ell^1(m_T)}\le
2C\sup_{A\in\mathcal{P}_F(\mathbb{N})}\Vert I_{m_T}(x\chi_A)\Vert_E,
\end{eqnarray*}
where in the last inequality we have used \eqref{EQ: L1(m)-norm}.
For every $A\in\mathcal{P}_F(\mathbb{N})$ we have that
$x\chi_A=\sum_{j\in A\cap M}x_je_j\in\ell^1$ and so
$I_{m_T}(x\chi_A)=T(x\chi_A)=\sum_{j\in A\cap M}x_jT(e_j)$. Then,
\begin{eqnarray*}
\Big\Vert\sum_{j\in M}x_j^r\,T(e_j)\Big\Vert_E^{\frac{1}{r}} & \le &
2C\sup_{A\in\mathcal{P}_F(\mathbb{N})}\Big\Vert\sum_{j\in A\cap
M}x_jT(e_j)\Big\Vert_E \\ & = & 2C\sup_{N\subset
M}\Big\Vert\sum_{j\in N}x_jT(e_j)\Big\Vert_E.
\end{eqnarray*}

Conversely, suppose that $T$ is $r$-power dominated and let
$x=(x_n)\in\ell^1(m_T)$. Taking
$y^k=\sum_{j=1}^k|x_j|^re_j\in\ell^1$, for every $k>\tilde{k}$ and
$A\in\mathcal{P}_F(\mathbb{N})$, we have that
$(y^k-y^{\tilde{k}})\chi_A=\sum_{j\in
A\cap\{\tilde{k}+1,...,k\}}|x_j|^re_j$ and so
\begin{eqnarray*}
\big\Vert T\big((y^k-y^{\tilde{k}})\chi_A\big)\big\Vert_E & = &
\Big\Vert \sum_{j\in
A\cap\{\tilde{k}+1,...,k\}}|x_j|^r\,T(e_j)\Big\Vert_E \\ & \le &
C^{\,r}\sup_{N\subset A\cap\{\tilde{k}+1,...,k\}}\Big\Vert\sum_{j\in
N}|x_j|T(e_j)\Big\Vert_E^r \\ & = & C^{\,r}\sup_{N\subset
A\cap\{\tilde{k}+1,...,k\}}\Big\Vert\,I_{m_T}\Big(\sum_{j\in
N}|x_j|e_j\Big)\Big\Vert_E^r \\ & \le & C^{\,r}\sup_{N\subset
A\cap\{\tilde{k}+1,...,k\}}\Big\Vert\sum_{j\in
N}|x_j|e_j\Big\Vert_{\ell^1(m_T)}^r \\ & \le &
C^{\,r}\big\Vert(y^k)^{\frac{1}{r}}-(y^{\tilde{k}})^{\frac{1}{r}}
\big\Vert_{\ell^1(m_T)}^r.
\end{eqnarray*}
For the last inequality note that
$(y^k)^{\frac{1}{r}}=\sum_{j=1}^k|x_j|e_j$ and so
$$
\sum_{j\in N}|x_j|e_j\le\sum_{j=\tilde{k}+1
}^k|x_j|e_j=(y^k)^{\frac{1}{r}}-(y^{\tilde{k}})^{\frac{1}{r}}
$$
for every $N\subset A\cap\{\tilde{k}+1,...,k\}$. Then, by using
\eqref{EQ: L1(m)-norm}, we have that
\begin{eqnarray*}
\big\Vert y^k-y^{\tilde{k}}\big\Vert_{\ell^1(m_T)} & \le &
2\sup_{A\in\mathcal{P}_F(\mathbb{N})}\big\Vert
I_{m_T}\big((y^k-y^{\tilde{k}})\chi_A\big)\big\Vert_E \\ & = &
2\sup_{A\in\mathcal{P}_F(\mathbb{N})}\big\Vert
T\big((y^k-y^{\tilde{k}})\chi_A\big)\big\Vert_E \\ & \le &
2C^{\,r}\big\Vert(y^k)^{\frac{1}{r}}-(y^{\tilde{k}})^{\frac{1}{r}}
\big\Vert_{\ell^1(m_T)}^r\to0
\end{eqnarray*}
as $k,\tilde{k}\to\infty$ since $(y^k)^{\frac{1}{r}}\uparrow|x|$
pointwise and $\ell^1(m_T)$ is $\sigma$-order continuous. Hence,
$y^k\to z$ in $\ell^1(m_T)$ for some $z\in\ell^1(m_T)$. In
particular, $y^k\to z$ pointwise and so $|x|^r=z\in\ell^1(m_T)$ as
$y^k\uparrow|x|^r$ pointwise. Therefore $x\in\ell^r(m_T)$.
\end{proof}

\begin{lemma}\label{LEM: rPowerDominated-rthFactorable}
Let $p>1$. If $T$ is $\frac{1}{p}$-power dominated then it is
$\frac{1}{p}$-th power factorable with a continuous extension.
\end{lemma}

\begin{proof}
Let us use Proposition \ref{PROP: T:l1->E}.(d). Given
$M\in\mathcal{P}_F(\mathbb{N})$ and $(x_j)_{j\in M}\subset
[0,\infty)$, denoting by $K$ the continuity constant of $T$, we have
that
\begin{eqnarray*}
\Big\Vert\sum_{j\in M}x_jT(e_j)\Big\Vert_E & = & \Big\Vert\sum_{j\in
M}(x_j^p)^{\frac{1}{p}}T(e_j)\Big\Vert_E\le
C^{\frac{1}{p}}\sup_{N\subset M}\Big\Vert\sum_{j\in
N}x_j^p\,T(e_j)\Big\Vert_E^{\frac{1}{p}} \\ & = &
C^{\frac{1}{p}}\sup_{N\subset M}\Big\Vert T\Big(\sum_{j\in
N}x_j^pe_j\Big)\Big\Vert_E^{\frac{1}{p}}\le
C^{\frac{1}{p}}K^{\frac{1}{p}}\sup_{N\subset
M}\Big\Vert\sum_{j\in N}x_j^pe_j\Big\Vert_{\ell^1}^{\frac{1}{p}} \\
& = & C^{\frac{1}{p}}K^{\frac{1}{p}}\sup_{N\subset M}\Big(\sum_{j\in
N}x_j^p\Big)^{\frac{1}{p}}\le
C^{\frac{1}{p}}K^{\frac{1}{p}}\Big(\sum_{j\in
M}x_j^p\Big)^{\frac{1}{p}}.
\end{eqnarray*}
\end{proof}

As a consequence of Remark \ref{REM:
l(1/p)(mT)Capl1(mT)-OptimalDomain}, Lemma \ref{LEM:
l1(mT)subsetlr(mT)} and Lemma \ref{LEM:
rPowerDominated-rthFactorable}, we obtain the following conclusion.

\begin{corollary}
For $p>1$ we have that:
\begin{itemize}\setlength{\leftskip}{-3ex}\setlength{\itemsep}{1ex}
\item[(a)] If $T$ is $p$-power dominated and $\frac{1}{p}$-th power
factorable with a continuous extension, then $T$ factors optimally
as
$$
\xymatrix{
\ell^1 \ar[rr]^{T} \ar@{.>}[dr]_(.35){i} & & E\\
& \ell^{\frac{1}{p}}(m_T) \ar@{.>}[ur]_(.6){I_{m_T}}}
$$
with $I_{m_T}$ being $\frac{1}{p}$-th power factorable with a
continuous extension.

\item[(b)] If $T$ is $\frac{1}{p}$-power dominated, then $T$ factors optimally
as
$$
\xymatrix{
\ell^1 \ar[rr]^{T} \ar@{.>}[dr]_(.35){i} & & E\\
& \ell^1(m_T) \ar@{.>}[ur]_(.6){I_{m_T}}}
$$
with $I_{m_T}$ being $\frac{1}{p}$-th power factorable with a
continuous extension.
\end{itemize}
\end{corollary}

Consider now the case when $E=\ell(c)$ is a B.f.s.\  related to $c$
such that $\ell^1\subset\ell(c)\subset\ell^0$. Then $\ell(c)$ is a
\emph{K\"{o}the function space} in the sense of Lindenstrauss and
Tzafriri, see \cite[p.\,28-30]{lindenstrauss-tzafriri}. For
instance, $\ell(c)$ could be an $\ell^q$ space with $1\le
q\le\infty$, or a Lorentz sequence space $\ell^{q,r}$ with $1\le
r\le q\le\infty$ or an Orlicz sequence space $\ell_\varphi$ with
$\varphi$ being an Orlicz function.

Let us recall some facts about the K\"{o}the dual of an space
$\ell(c)$. Denote the \emph{scalar product} of two sequences
$x=(x_n),y=(y_n)\in\ell^0$ by
$$
\big(x,y\big)=\sum x_ny_n
$$
provided the sum exists. The \emph{K\"{o}the dual} of $\ell(c)$ is
given by
$$
\ell(c)'=\Big\{y\in\ell^0:\,\big(|x|,|y|\big)<\infty \ \textnormal{
for all } x\in\ell(c)\Big\}.
$$
Note that $\chi_A\in\ell(c)'$ for all
$A\in\mathcal{P}_F(\mathbb{N})$. The space $\ell(c)'$ endowed with
the norm
$$
\Vert y\Vert_{\ell(c)'}=\sup_{x\in B_{\ell(c)}}\big(|x|,|y|\big)
$$
is a B.f.s.\ in the sense of Lindenstrauss and Tzafriri. The map
$j\colon\ell(c)'\to\ell(c)^*$ defined by $\langle
j(y),x\rangle=\big(x,y\big)$ for all $y\in\ell(c)'$ and
$x\in\ell(c)$, is a linear isometry. In particular, convergence in
norm of $\ell(c)$ implies pointwise convergence, as $e_n\in\ell(c)'$
for all $n$. Note that $\ell(c)\subset\ell(c)''$. The equality
$\ell(c)=\ell(c)''$ holds with equal norms if and only if $\ell(c)$
has the \emph{Fatou property}, that is, if $(x^k)\subset\ell(c)$ is
such that $0\le x^k\uparrow x$ pointwise and $\sup\Vert
x^k\Vert_{\ell(c)}<\infty$ then $x\in\ell(c)$ and $\Vert
x^k\Vert_{\ell(c)}\uparrow\Vert x\Vert_{\ell(c)}$.

Let $M=(a_{ij})$ be an infinite matrix of real numbers and denote by
$C_j$ the $j$-th column of $M$. Assume $C_j\not=0$ for all $j$. Note
that
$$
Mx=\Big(\sum_ja_{ij}x_j\Big)_i
$$
for any $x\in\ell^0$ for which it is meaningful to do so.

\begin{proposition}\label{PROP: M:l1->l(c)}
Suppose that $\ell(c)$ has the Fatou property. Then, the following
statements are equivalent:
\begin{itemize}\setlength{\leftskip}{-3ex}\setlength{\itemsep}{.5ex}
\item[(a)] $M$ defines a continuous linear operator
$M\colon\ell^1\to\ell(c)$.

\item[(b)] $C_j\in\ell(c)$ for all $j$ and $\sup_j\Vert C_j\Vert_{\ell(c)}<\infty$.
\end{itemize}
\end{proposition}

\begin{proof}
(a) $\Rightarrow$ (b) Let $K>0$ be such that $\Vert
Mx\Vert_{\ell(c)}\le K\Vert x\Vert_{\ell^1}$ for all $x\in\ell^1$.
For every $j$ we have that $C_j=Me_j\in\ell(c)$. Moreover,
$$
\sup_j\Vert C_j\Vert_{\ell(c)}=\sup_j\Vert Me_j\Vert_{\ell(c)}\le
K\sup_j\Vert e_j\Vert_{\ell^1}=K.
$$

(b) $\Rightarrow$ (c) Since $\ell(c)$ has the Fatou property then
$\ell(c)=\ell(c)''$ with equal norms. Let $x\in\ell^1$. First note
that for every $i$ we have that
\begin{eqnarray*}
\sum_j|a_{ij}x_j| & = & \sum_j\big(|C_j|,e_i\big)|x_j|\le\sum_j\Vert
C_j\Vert_{\ell(c)}\Vert e_i\Vert_{\ell(c)'}|x_j| \\ & \le & \Vert
e_i\Vert_{\ell(c)'}\Vert x\Vert_{\ell^1}\sup_j\Vert
C_j\Vert_{\ell(c)}
\end{eqnarray*}
and so $Mx\in\ell^0$. Given $y\in\ell(c)'$ it follows that
\begin{eqnarray*}
\big(|y|,|Mx|\big) & = & \sum_i|y_i|\Big|\sum_ja_{ij}x_j\Big|\le
\sum_i\sum_j|a_{ij}x_jy_i|=\sum_j|x_j|\sum_i|a_{ij}y_i|
\\ & = & \sum_j|x_j|\big(|C_j|,|y|\big)
\le\sum_j|x_j|\,\Vert C_j\Vert_{\ell(c)}\Vert y\Vert_{\ell(c)'} \\ &
\le & \Vert y\Vert_{\ell(c)'}\Vert x\Vert_{\ell^1}\sup_j\Vert
C_j\Vert_{\ell(c)}.
\end{eqnarray*}
Then $Mx\in\ell(c)''=\ell(c)$ and
$$
\Vert Mx\Vert_{\ell(c)}=\sup_{y\in
B_{\ell(c)'}}\big(|y|,|Mx|\big)\le\Vert x\Vert_{\ell^1}\sup_j\Vert
C_j\Vert_{\ell(c)}.
$$
\end{proof}

In what follows assume that $\ell(c)$ has the Fatou property,
$C_j\in\ell(c)$ for all $j$ and $\sup_j\Vert
C_j\Vert_{\ell(c)}<\infty$. Then, $M$ defines a continuous linear
operator $M\colon\ell^1\to\ell(c)$ and so, by Remark \ref{REM:
l1(mT)-OptimalDomain} we have that $M$ can be extended as
$$
\xymatrix{
\ell^1 \ar[rr]^{M} \ar@{.>}[dr]_(.4){i} & & \ell(c)\\
& \ell^1(m_M) \ar@{.>}[ur]_(.55){I_{m_M}}}
$$
and $\ell^1(m_M)$ is the largest $\sigma$-order continuous
quasi-B.f.s.\ to which $M$ can be extended as a continuous operator.

\begin{remark}\label{REM: ImM=M}
For every $x\in\ell^1(m_M)$ it follows that $I_{m_M}(x)=Mx$ and so
$M$ defines a continuous linear operator
$M\colon\ell^1(m_M)\to\ell(c)$. Indeed, take $0\le
x=(x_n)\in\ell^1(m_M)$ and $x^k=\sum_{j=1}^kx_je_j\in\ell^1$. Since
$x^k\uparrow x$ pointwise and $\ell^1(m_M)$ is $\sigma$-order
continuous it follows that $x^k\to x$ in $\ell^1(m_M)$. Then, since
$M=I_{m_M}$ on $\ell^1$, we have that $Mx^k=I_{m_M}(x^k)\to
I_{m_M}(x)$ in $\ell(c)$ and so pointwise. Hence, the $i$-th
coordinate $\sum_{j=1}^ka_{ij}x_j$ of $Mx^k$ converges to the $i$-th
coordinate of $I_{m_M}(x)$ and thus $Mx=I_{m_M}(x)\in\ell(c)$. For a
general $x\in\ell^1(m_M)$, we only have to take the positive and
negative parts of $x$.
\end{remark}

From Proposition \ref{PROP: T:l1->E} applied to
$M\colon\ell^1\to\ell(c)$ and Remark \ref{REM: ImM=M} we obtain the
following conclusion.

\begin{proposition}\label{PROP: M:l1->l(c)}
The following statements are equivalent:
\begin{itemize}\setlength{\leftskip}{-3ex}\setlength{\itemsep}{1ex}
\item[(a)] $M$ defines a continuous linear operator
$M\colon\ell^p\to\ell(c)$.

\item[(b)] $M$ is $\frac{1}{p}$-th power factorable with a continuous extension.

\item[(c)] $\ell^p\subset\ell^1(m_M)$.

\item[(d)] $\ell^1\subset\ell^{\frac{1}{p}}(m_M)\cap\ell^1(m_M)$.

\item[(e)] There exists $C>0$ such that
$$
\Big\Vert\sum_{j\in M}x_jC_j\Big\Vert_{\ell(c)}\le C\Big(\sum_{j\in
M}x_j^p\Big)^{\frac{1}{p}}
$$
for all $M\in\mathcal{P}_F(\mathbb{N})$ and $(x_j)_{j\in M}\subset
[0,\infty)$.
\end{itemize}
\end{proposition}

\begin{proof}
The equivalence among statements (b), (c), (d), (e) is given by
Proposition \ref{PROP: T:l1->E}. The stamement (a) implies (b)
obviously. From Remark \ref{REM: ImM=M} we have that $M$ defines a
continuous linear operator $M\colon\ell^1(m_M)\to\ell(c)$, so (c)
implies (a).
\end{proof}

Let us give two conditions guaranteeing that $M$ defines a
continuous linear operator $M\colon\ell^p\to\ell(c)$:
\begin{itemize}\setlength{\leftskip}{-3ex}\setlength{\itemsep}{1ex}
\item[(I)] If $p'$ is the conjugate exponent of $p$ and $\sum\Vert
C_j\Vert_{\ell(c)}^{p'}<\infty$, then (e) in Proposition \ref{PROP:
M:l1->l(c)} holds. Indeed, for every $M\in\mathcal{P}_F(\mathbb{N})$
and $(x_j)_{j\in M}\subset [0,\infty)$ we have that
\begin{eqnarray*}
\Big\Vert\sum_{j\in M}x_jC_j\Big\Vert_{\ell(c)} & \le & \sum_{j\in
M}x_j\Vert C_j\Vert_{\ell(c)}\le\Big(\sum_{j\in
M}x_j^p\Big)^{\frac{1}{p}}\Big(\sum_{j\in M}\Vert
C_j\Vert_{\ell(c)}^{p'}\Big)^{\frac{1}{p'}} \\ & \le &
\Big(\sum\Vert
C_j\Vert_{\ell(c)}^{p'}\Big)^{\frac{1}{p'}}\Big(\sum_{j\in
M}x_j^p\Big)^{\frac{1}{p}}.
\end{eqnarray*}

\item[(II)] If $M$ is $\frac{1}{p}$-power dominated, that is,
there exists $C>0$ such that
$$
\Big\Vert\sum_{j\in M}x_j^{\frac{1}{p}}\,C_j\Big\Vert_{\ell(c)}^p\le
C\sup_{N\subset M}\Big\Vert\sum_{j\in N}x_jC_j\Big\Vert_{\ell(c)}
$$
for every $M\in\mathcal{P}_F(\mathbb{N})$ and $(x_j)_{j\in
M}\in[0,\infty)$, then (b) in Proposition \ref{PROP: M:l1->l(c)}
holds by Lemma \ref{LEM: rPowerDominated-rthFactorable}.
\end{itemize}

For instance, in the case when $\ell(c)=\ell^q$ and $a_{ij}\ge0$ for
all $i,j$, condition (II) is satisfied if $F_i\in\ell^1$ for all $i$
and $\sum\Vert F_i\Vert_{\ell^1}^q<\infty$, where $F_i$ denotes the
$i$-th file of $M$. Indeed, for every
$M\in\mathcal{P}_F(\mathbb{N})$ and $(x_j)_{j\in M}\in[0,\infty)$,
applying H\"{o}lder's inequality twice for $p$ and its conjugate
exponent $p'$, we have that
\begin{eqnarray*}
\Big\Vert\sum_{j\in M}x_j^{\frac{1}{p}}\,C_j\Big\Vert_{\ell^q}^p & =
& \Big(\sum_{i}\Big(\sum_{j\in
M}x_j^{\frac{1}{p}}\,a_{ij}\Big)^q\Big)^{\frac{p}{q}}=
\Big(\sum_{i}\Big(\sum_{j\in
M}x_j^{\frac{1}{p}}\,a_{ij}^{\frac{1}{p}}a_{ij}^{1-\frac{1}{p}}\Big)^q\Big)^{\frac{p}{q}}
\\ & \le & \Big(\sum_{i}\Big(\sum_{j\in
M}x_j\,a_{ij}\Big)^{\frac{q}{p}}\Big(\sum_{j\in
M}a_{ij}\Big)^{\frac{q}{p'}}\Big)^{\frac{p}{q}}
\\ & \le & \Big(\sum_{i}\Big(\sum_{j\in
M}x_j\,a_{ij}\Big)^q\Big)^{\frac{1}{q}}\cdot \Big(\sum_{i}
\Big(\sum_{j\in M}a_{ij}\Big)^q\Big)^{\frac{p}{qp'}} \\ & \le &
\Big\Vert\sum_{j\in M}x_j\,C_j\Big\Vert_{\ell^q}\Big(\sum_i\Vert
F_i\Vert_{\ell^1}^q\Big)^{\frac{p}{qp'}}.
\end{eqnarray*}
Note that $\sup_{N\subset M}\Big\Vert\sum_{j\in
N}x_jC_j\Big\Vert_{\ell^q}=\Big\Vert\sum_{j\in
M}x_j\,C_j\Big\Vert_{\ell^q}$ as $a_{ij}\ge0$ for all $i,j$.


\end{document}